\documentclass[reqno]{amsart}
\usepackage{latexsym,amssymb,amsthm,amsmath}
\usepackage{graphicx}

\usepackage[bookmarksnumbered, colorlinks]{hyperref}
\usepackage{amssymb}
\usepackage{amsmath}
\title[Fibered boundary,Seiberg-Witten Equations]{Inequalities for 4-Manifolds with Ends using Fibered boundary metric and Seiberg-Witten Equations}

\author[Shyamal Kumar Hui and Chandan Kumar Adak]{Shyamal Kumar Hui and Chandan Kumar Adak}


\usepackage{multicol}
\theoremstyle{plain}
\newtheorem{theorem}{Theorem}[section]
\newtheorem*{Theorem B}{Theorem B}
\newtheorem*{Theorem A}{Theorem A}
\newtheorem*{Theorem C}{Theorem C}
\newtheorem{lemma}[theorem]{Lemma}
\newtheorem{proposition}[theorem]{Proposition}
\newtheorem{corollary}[theorem]{Corollary}

\numberwithin{equation}{section}

\theoremstyle{remark}

\begin{document}

\maketitle
\begin{abstract}
	 We deduce a new inequality using the Hitchin-Thorpe inequality for noncompact complete four-dimensional manifolds with fibered geometry at infinity and Seiberg-Witten gauge-theoretic techniques. Fibered geometry at infinity is either fibered boundary metric or fibered cusp metric. We prove it for fibered boundary metric when the fiber at boundary is a single point.
\end{abstract}
\section{Introduction}
Weighted manifolds or manifolds with density is a Riemannian manifold $(M,g)$ having a function $f:M\to \mathbb{R}$, giving the measure $e^{-f}dV_g$, where $dV_g$ is the volume form \cite{Bal-Ozu-1}.

A weighted manifold $M$ is said to be a manifold with finite ends $M_1,M_2,M_3,\ldots,M_k$ if there exists a compact subset $C$ of $M$ such that $M\backslash C$ has $k$ connected components $P_1,P_2,P_3,\ldots,P_k$ where each $P_i$ is isometric to $M_i\backslash C_i$ for some compact subset $C_i$ of $M_i$ and $M_i$'s are non-compact, geodesically complete weighted manifolds of the same dimension. Each $P_i$(or $M_i$) will be considered as end of $M$. For more discussions on manifold with ends, see \cite{Grig_Ishi_Salo-1}.
\begin{figure}[h]
    \centering
    \includegraphics[width=0.7\linewidth]{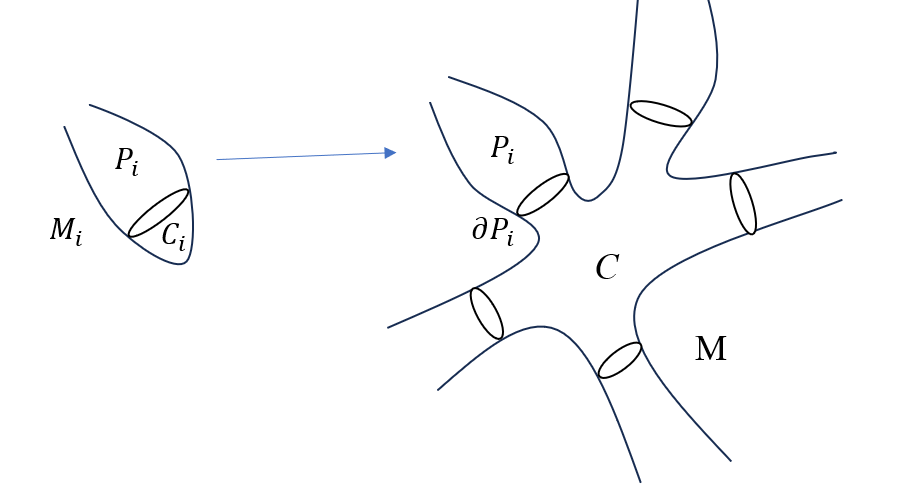}
    \caption{Manifold with ends}
   
\end{figure}

A manifold with cusps is a special type of manifold with ends where each end is homeomorphic (or diffeomorphic) to $S\times [0,\infty)$ with a warped product metric of the form $ds^2=dt^2+e^{-2t}g_S$ where $g_S$ is a flat metric on the compact cross-section $S$. Taking the transformation $y=e^t$, this metric can be written as $$ds^2=\frac{dy^2+g_S}{y^2}$$ As $t\to \infty$, presence of the factor $e^{-2t}$ ensures that the cross-sectional area decays exponentially and the volume of the cusp end becomes finite.

Let $X$ be a noncompact oriented Riemannian 4-manifold with cylindrical end all of the form $T^3\times [0,\infty)$, where $T^3$ is the 3-Torus. We can construct a hyperbolic metric on $T^3\times [0,\infty)$ as : $$g_{hyp}=dt^2+e^{-2t}g_{T^3}$$ where $g_{T^3}$ is a flat metric on $T^3$ and $t\in[0,\infty)$. A metric $g$ on $X$ is said to be asymptotically hyperbolic if $g$ is asymptotically $C^2$-close to $g_{hyp}$ for some flat metric $g_{T^3}$ on each end. For further study on the asymptotically hyperbolic metric of order $C^{m,\alpha}$, see \cite{Lee1}. In general, an asymptotically hyperbolic metric comes from a hyperbolic 4-manifold. Long-Reid \cite{Long-Reid1}  showed that the signature of all finite-volume hyperbolic 4-manifolds with $T^3$ cusps would be zero. Again, noncompact orientable Riemannian 4-manifolds with cylindrical ends appear as link complements. For a detailed study on link complements, see \cite{Iva1,Irt1,Sar1}. Anderson \cite{And1} also constructed many examples of asymptotically hyperbolic Einstein metrics via Dehn filling. This type of construction deals with the existence of asymptotically hyperbolic cusp metrics.

Let $(X,g)$ be a Riemannian manifold. Then $g$ is said to be Einstein metric if its Ricci tensor satifies the condition $$Ric_g=\lambda g$$ for some scalar $\lambda$. A breakthrough result for the Riemannian manifold is the Hitchin-Thorpe inequality, which states that if a compact, oriented 4-manifold $X$ possesses an Einstein metric $g$, then its Euler number $\chi(X)$ and signature $\sigma (X)$ satisfy the following inequality $$\chi(X)\geq \frac{3}{2}|\sigma(X)|.$$ This is a necessary condition for existence of Einstein metric for compact oriented 4-manifold but not sufficient. In the mid-1990s, LeBrun \cite{LEB1,LEB3} and Sambusetti \cite{SAMBU} independently showed the existence of a 4-manifold satisfying the Hitchin-Thorpe inequality, but admits no Einstein metrics. LeBrun \cite{LEB1,LEB3} constructed it using techniques from Seiberg-Witten theory.

Manifold with fibered geometry at infinity is a special case of a manifold with cusps and there are many generalisations of the Hitchin-Thorpe inequality for manifold with fibered geometry at infinity; for detailed observation, see \cite{KOT1,KOT2,GRO1,SAMBU}.  The asymptotic geometry at infinity is either a fiber bundle over a cone with a compact fiber, called a fibered boundary, or a cusp bundle over a compact space, called fibered cusps.  Dai-Wei \cite{DAI-WEI1} proved Hitchin-Thorpe inequality for Einstein noncompact 4-manifolds with specified geometry at infinity as follows:
\begin{theorem}
 \cite{DAI-WEI1} Let $(M^4,g)$ be a noncompact complete Einstein manifold which is asymptotic to a fibered cusp or a fibered boundary at infinity. In the fibered boundary case, we also assume that $dim\, F>0$. Then 
    $$\chi (M)\geq \frac{3}{2}\left|\sigma(M)+\frac{1}{2}a-\lim \eta\right|,$$ where $a-\lim \eta$ is the adiabatic limit of the eta invariant of $\partial \Bar{M}.$  Moreover, the equality holds iff $(M,g)$ is a complete Calabi-Yau manifold.
\end{theorem}
The adiabatic limit of the eta invariant of $\partial \Bar{M}$ tells about many geometric and topological properties of fibered geometry at infinity \cite{DAI1}. Again, if the fiber is a single point, Dai-Wei \cite{DAI-WEI1} proved another inequality as follows:
\begin{theorem}
  \cite{DAI-WEI1} Let $(M,g)$ be a complete Einstein four-manifold which is asymptotic to a cone over $(\partial \Bar{M},g_{\partial \Bar{M}})$. Then 
    \begin{eqnarray*}
        2\chi(M) \geq 3\left| \frac{1}{2}\eta(\partial\Bar{M})+ \sigma(M) \right|+\frac{1}{\pi^2}vol(\partial\Bar{M}+2\alpha(\partial \Bar{M})
    \end{eqnarray*}
    where $\eta(\partial \Bar{M})$ is the eta invariant of $(\partial \Bar{M},g_{\partial \Bar{M}})$ and $\alpha(\partial \Bar{M})$ a geometric invariant defined by $$\alpha(\partial \Bar{M})=\frac{1}{8\pi^2}\int_{\partial \Bar{M}}\epsilon_{abc}\omega^a \wedge\left(\Omega_c^b-\omega^b\wedge \omega^c\right)$$, $\Bar{M}=M\cup \partial \Bar{M}$ is the compactification of $M$, $\omega^a$ is the dual 1-forms of an orthonormal basis on $\partial \Bar{M}$, and $\Omega^b_c$ is the curvature form of $\partial \Bar{M}$ w.r.t the orthonormal basis . Moreover, the equality holds iff $M$ is an asymptotically conical Calabi-Yau manifold.
\end{theorem}

Goal of this paper is to find a new form of the inequality on manifolds with fibered geometry at infinity using earlier work of Dai-Wei \cite{DAI-WEI1} and Seiberg-Witten gauge theoretic techniques.

\section{Prelimineries}

In this section, we discuss fibered geometry at infinity, Euler characteristics, signature, $ L^2$-cohomology, $L^2$-Chern-Weil theory, and Seiberg-Witten equations.

\subsection{Fibered Geometry at Infinity}
Let $(M,g)$ be a complete noncompact Riemannian $n$-manifold with finite topological type and $\Bar{M}=M \cup \partial\Bar{M}$ be its compactification. Let $$F\to \partial \Bar{M} \xrightarrow{p} B$$ be a fibration structure on the boundary, where $B, F$ are closed manifolds. Let $[0,\epsilon)\times \partial \Bar{M}$ be a trivialization of a neighbourhood of $\partial \Bar{M}$, where $x\in C^{\infty}(\Bar{M})$ is a coordinate in $[0,\epsilon)$ with $x=0$ on $\partial\Bar{M}$ and $x>0$ in $M$; also $dx$ is nowhere vanishing on $\partial\Bar{M}.$ The metric $g$ is said to be asymptotic to a fibered cusp if its local form near the boundary $\partial \Bar{M}$ takes the form 
\begin{equation}\label{EQ-2.1}
    g \sim \frac{dx^2}{x^2}+p^*g_{B}+x^2g_{F}
\end{equation}
where $g_F$ is a family of metrics along the fibers and $g_B$ is a metric on the base manifold $B$. Putting $x=e^{-t}$ in (\ref{EQ-2.1}), we get 
\begin{equation}\label{EQ-2.2}
    g\sim dt^2+p^*g_B+e^{-2t}g_F.
\end{equation}
Thus we observe that for the fibered cusp case fibers shrink to zero, but the base remains fixed as $x\to 0.$

The metric $g$ is said to be asymptotic to a fibered boundary metric if its local form near the boundary takes the form 
\begin{equation}\label{EQ-2.3}
    g\sim \frac{dx^2}{x^4}+\frac{p^*g_B}{x^2}+g_F.
\end{equation}
Putting $x=\frac{1}{t}$ in (\ref{EQ-2.3}), we get
\begin{equation}\label{EQ-2.4}
    g\sim dt^2+t^2p^*g_B+g_F.
\end{equation}
For this case, we observe that fiberes remain bounded, but the base expands as $x\to 0$. For a detailed discussion, see \cite{GRI-SAN-VER,DAI-WEI1,VAL1}.

Now, we consider a Lie algebra of vector fields on $\Bar{M}$ as follows:
$$^{\phi}\mathcal{V}(\Bar{M})=\{X\in T(\Bar{M}): X \text{ is tangent to the fibers at the boundary, and }X(x)=O(x^2) \}.$$

This $^{\phi}\mathcal{V}(\Bar{M})$ induces a vector bundle structure $^{\phi}T\Bar{M}$ on $\Bar{M}$ such that $$^{\phi}\mathcal{V}(\Bar{M})=\Gamma (^{\phi}T\Bar{M})$$ and local frame of $^{\phi}T(\Bar{M})$ near $\partial \Bar{M}$ is given by $x^2\partial_x,x\partial_y,\partial_z$ where $y,z$ are local coordinates of base $B$ and fiber $F$ respectively. Again, $$\text{End}(^{\phi}T\Bar{M})|_M \cong \text{End}(\text{TM})$$ and different isomorphisms differ by conjugation, see \cite{VAL1}. This canonically identifies invariant polynomials, such as trace polynomials.

A metric $g_1$ is said to be asymptotic to a fiber-boundary metric if 
\begin{equation}\label{EQ-2.5}
    g_1=g+xs
\end{equation}
where $g$ is a fiber-boundary metric as defined in (\ref{EQ-2.3}) and $s$ is a smooth symmetric two-tensor i.e.  $s\in \mathcal{S}^2(^{\phi}T\Bar{M})$ with $s(x^2\partial_x,\cdot)\equiv 0$, where $\mathcal{S}^2$ denotes the space of symmetric tensors.

The following proposition states about the nature of Levi-Civita connection for a metric asymptotic to the fibered boundary metric.
\begin{proposition}
    The Levi-Civita connection for a metric asymptotic to the fibered boundary metric is a true connection, i.e., $$\nabla ^{\phi}:\Gamma(^{\phi}T\Bar{M})\to \Gamma (T^{*}\Bar{M}\otimes ^{\phi}T\Bar{M}).$$
    Moreover, $$R^{\phi}\in\Gamma(\Lambda ^2T^*\Bar{M}\otimes \text{End}(^{\phi}T\Bar{M})).$$
\end{proposition}
For detailed proof, see \cite{VAL1}.
\subsection{Characteristic and Signature} 
For an oriented Riemannian bundle $E\to M$ of rank $2m$, the global class $e(E)=[\text{Pf}(\frac{\Omega}{2\pi})]\in H^{2m}(M)$ is called the \textbf{Euler class} of the oriented Riemannian bundle $E$, where $\text{Pf}(X)$ is the $\textbf{Pfaffian}$ polynomial of degree $m$ and $\Omega$ is the curvature matrix relative to a positively oriented orthonormal frame of $E\to M$. We get the generalized Gauss-Bonnet theorem as follows:
\begin{theorem}
    Let $M$ be a compact oriented Riemannian manifold of dimension 2m, and $\nabla$ be a metric connection on the tangent bundle $TM$ with curvature matrix $\Omega$ relative to a positively oriented orthonormal frame. Then $$\chi (M)=\int_M e(TM)=\int_M \text{Pf} \left(\frac{1}{2\pi}\Omega\right).$$
\end{theorem}
For $2m=4$, we explicitly get
\begin{eqnarray}
   \nonumber \chi(M) & = & \frac{1}{32\pi^2}\int_M \epsilon_{abcd}\Omega_{ab}\wedge\Omega_{cd} \\
  \label{EQ-26} & = & \frac{1}{8\pi^2}\int_M \left(|W|^2 - |Z|^2+\frac{1}{24}S^2\right)d\text{vol}.
\end{eqnarray}
Here $\epsilon_{abcd}$ is the sign of the permutation $\rho$ s.t. $\rho(1)=a,\ldots,\rho(4)=d,$ $W$ is the Weyl curvature such that $|W|^2=|W^+|^2+|W^-|^2$ for Einstein manifolds, here $W^+,\, W^-$ represent self-dual and anti-self dual components of the Weyl curvature tensor $W$, $Z$ the traceless Ricci, and $S$ the scalar curvature \cite{MIL-STA1, TU1}.

On a $4n-$dimensional compact oriented manifold $M$, a symmetric bilinear pairing $$\langle ,\rangle:H^{2n}(M;\mathbb{R})\times H^{2n}(M;\mathbb{R}) \to H^{4n}(M;\mathbb{R})\cong \mathbb{R}, $$ is called the \textbf{intersection form}. This intersection form can be represented by a symmetric matrix for which all of its eigenvalues are real. If $ b^+$ and $ b^-$ are respectively positive and negative eigenvalues of the symmetric matrix, then $\sigma(M):=b^+-b^-$ is called the \textbf{signature} of $M$.

In 1953, Hirzebruch established a formula for the signature as follows:
\begin{theorem}[Hirzebruch signature formula]
   Let $M$ be a compact oriented smooth manifold of dimension $4n$. Then its signature $\sigma(M)$ is given by $$\sigma(M)=\int_M L_n(p_1,\ldots,p_n),$$ where the $L_n$'s are $L-$polynomials and $p_i$'s are Pontrjagin classes.
\end{theorem}
In dimension 4, we explicitly get 
\begin{eqnarray}
   \nonumber \sigma(M) &=&\int_M L\left(\frac{\Omega}{2\pi}\right)\\
   \nonumber &=& -\frac{1}{24\pi^2}\int_M \text{Tr}(\Omega \wedge \Omega)\\
  \label{EQ-27}  &=& \frac{1}{12\pi^2}\int_M \left(|W^+|^2-|W^-|^2\right) d\text{vol}.
\end{eqnarray}
where $W^+,\, W^-$ represent self-dual and anti-self dual components of the Weyl curvature tensor $W$ \cite{MIL-STA1, TU1}.

Generalized Gauss-Bonnet theorem and the Hirzebruch signature formula are both particular forms of the Atiyah-Singer index theorem. Later, we will use the Atiyah-Patodi-Singer index theorem \cite{APS1} for manifolds with boundary.

\subsection{Seiberg-Witten Equations}
A spinor bundle over a Riemannian manifold $M$ is a Hermitian vector bundle $S$ of complex rank $2^m$, with a map $\Gamma:TM \to End\, S$ satisfying the condition $\Gamma(v)+\Gamma^*(v)=0$ and $\Gamma(v)\Gamma^*(v)=-|v|^2Id$, for all $v\in TM.$ A spin connection on $S$ is a connection $$\nabla : C^{\infty}(S) \to C^{\infty}(T^*M\otimes S)$$ compatible with the Levi-Civita connection $\nabla^{LC}$ such that for any section $s$ of $S$ and any two vector fields $u,\, v$ on $M$, we have $$\nabla_v(u\cdot s)=\nabla_v^{LC}(u)\cdot s+ u\cdot \nabla_v(s).$$

The Seiberg-Witten equations are given in terms of a pair $(A,\phi)$, where $A$ is a spin connection and $\phi$ is a section of $S^+$, where $S=S^+\oplus S^-$. Let $D_A$ be the dirac operator correspontiong to $A$, $\hat{A}$ and $F_{\hat{A}}$ are induced connection and curvature on the line bundle $L$, respectively, where $L$ is the determinat line bundles of $S^+$ and $S^-$ and it induces same connection on both the spaces. The equations are 
\begin{eqnarray}
    D_A\phi=0 \\ F^+_{\hat{A}}=\frac{1}{4}\langle e_i e_j\phi,\phi\rangle e^i\wedge e^j,
\end{eqnarray}
where $\langle,\rangle$ is the inner product on the fibers of $S^+$, $\{e_i\}$ is a local basis of $TX$ acting on $\phi$ by clifford multiplication, and $\{e^i\}$ is a local basis of $T^*X$. For more discussion, see \cite{SAL1,BOO-BLEE1,MAR1}.

\subsection{$L^2$ Cohomology and $L^2$ Chern-Weil theory}
Suppose $(M,g)$ is an oriented noncompact $n$-manifold and $L^2\Omega^k(M,g)$ denotes the $ L^2$-completion of $ L^2$-norm bounded smooth $k$-forms with respect to $g$. If we restrict the deRham differential $d$ to $L^2\Omega^k(M,g)$, we have a Hilbert complex $$\cdots \to L^2\Omega^{k-1}(M,g) \to L^2\Omega^{k}(M,g) \to L^2\Omega^{k+1}(M,g) \to \cdots$$
If $Dom^k(d)=\{\alpha\in L^2\Omega^{k}(M,g)|d\alpha \in L^2\Omega^{k+1}(m,g)\}$ and $Z^k(M,g)=\{\alpha\in L^2\Omega^{k-1}(M,g)|d\alpha =0\}$, the $k$th $L^2-$cohomology is defined as $$H^k_{L^2}(M,g)=Z^k(M,g)/\overline{d\,Dom^{k-1}(d)}.$$

Suppose $(M,g)$ is a complete, finite volume 4-manifold and $\mathcal{L}$ is a complex line bundle over $M$, and $A$ is a connection on $\mathcal{L}$ such that curvature form $F_A\in L^2\Omega^{k}(M,g).$ Then the $L^2$ Chern class of $\mathcal{L}$ is defined as $$c_1(\mathcal{L})=\frac{i}{2\pi}[F_A]_{L^2}.$$
For more detailed discussion, see \cite{Xu1}. 

\section{Main Results}
In this section, we discussed on Chern-Simons correction term, results on Seiberg-Witten Gauge theory, and their applications using fibered geometry at infinity.

\subsection{Chern-Simons correction term}
Suppose $(M,g)$ is a complete noncompact manifold with fibered geometry at infinity and $M_{\epsilon}=\{x\geq \epsilon\}$ with $\partial M_{\epsilon}=\{x=\epsilon\}$ for sufficiently small positive $\epsilon$. The metric on $M_{\epsilon}$ is not product-type near the boundary, so the Chern-Simons term arises. 

The Euler characteristic and signature of such manifolds are determined by the Atiyah-Patodi-Singer index formula \cite{APS1}. By their topological nature, we have $$\chi(M)=\chi(M_{\epsilon}),\,\,\,\, \sigma(M)=\sigma(M_{\epsilon}).$$

Let $P$ be a homogeneous polynomial of degree $r$. Let $\omega$ and $\omega '$ be two connections on the bundle and $\Omega$ and $\Omega '$ their curvatures. Then $P(\Omega)$ defines a global form. Let $\omega_t=\omega +t(\omega ' - \omega), 0\leq t \leq 1,$ be the interpolation of $\omega,\omega '$ and its curvature $\Omega_t=\text{d}\omega_t +\omega_t \wedge \omega_t$. Then we have 
\begin{equation}\label{EQ-31}
    P(\Omega ') - P(\Omega)=\text{d}\int_0^1 rP(\omega '-\omega,\underbrace{\Omega_t,\ldots,\Omega_t}_{r-1})=\text{d}Q(\omega ',\omega)
\end{equation}
where $$\int_0^1rP(\omega '-\omega,\underbrace{\Omega_t,\ldots,\Omega_t}_{r-1})=Q(\omega ',\omega).$$ For detailed discussion, see \cite{EGH1}.

If $N$ is a compact oriented 4-dimensional Riemannian manifold with boundary $\partial N$ and it has a metric which is product type near the boundary, then $$\chi(N)=\int_N\text{Pf}\left(\frac{\Omega}{2\pi}\right), $$ and $$\sigma(N)=\int_N L\left(\frac{\Omega}{2\pi}\right) - \frac{1}{2}\eta(\partial N),$$ where $\eta(\partial N)$ represents the eta invariant of the signature operator on the boundary. 

Let $g$ be a metric on $N$ which is not product type near the boundary $\partial N$; then its expression near the boundary will be of the form $$g=dt^2+\gamma(t),$$ where $t$ is the distance from the boundary and $\gamma(t)$ is the restriction of $g$ on $\partial N$. Let $g'$ be another metric on $N$ that equals $g$ except near the boundary, but is of product type near the boundary. Suppose that $\omega$ and $\omega '$ are the connection forms of $g$ and $g '$, respectively, and $\Omega$ and $\Omega '$ their respective curvatures. Then the Chern-Simons correction term to the Atiyah-Patodi-Singer index theorem is given by, using (\ref{EQ-31}), 
\begin{equation}\label{EQ-2.6}
    \int_N P(\Omega ') - \int_N P(\Omega) = \int_N \text{d}Q(\omega ',\omega)=\int_{\partial N} Q(\omega ',\omega)
\end{equation}
where $$\int_0^1rP(\xi,\underbrace{\Omega_t,\ldots,\Omega_t}_{r-1})=Q(\omega ',\omega),$$ where $\xi = \omega - \omega '$ is the second fundamental form at the boundary. For detailed discussion, see \cite{APS1,EGH1}. 
 Thus, the new expression of the Euler characteristic and signature containing the Chern-Simons correction term will be 
\begin{equation}\label{EQ-3.3}
    \chi(N)=\int_N\text{Pf}\left(\frac{\Omega}{2\pi}\right) - \int_{\partial N} Q_{\chi}(\omega ',\omega),
\end{equation}
and
\begin{equation}\label{EQ-3.4}
    \sigma(N)=\int_N L\left(\frac{\Omega}{2\pi}\right) - \frac{1}{2}\eta(\partial N) - \int_{\partial N} Q_{\sigma}(\omega ',\omega)
\end{equation}
where $Q_{\chi}$ and $Q_{\sigma}$ are Chern-Simons correction terms for Euler characteristics and signature, respectively.

Let $g_1$ and $g$ be two metrics such that $g_1=g+xs$, where $s$ is a symmetric smooth two-tensor satisfying $s(x^2\partial_x ,\cdot)\equiv 0$. Then we have the following lemma:
\begin{lemma}\label{LM-3.1}
    Let $Q_1,Q$ be the Chern-Simons correction term corresponding to $g_1,g$, respectively. Then we have $$\lim_{\epsilon \to 0}\int_{\partial M_{\epsilon}}Q_1 = \lim_{\epsilon \to 0} \int_{\partial M_{\epsilon}}Q$$
\end{lemma}
\begin{proof}
    See \cite{DAI-WEI1}.
\end{proof}

If $P$ is Pfaffian, then explicit form of $Q_{\chi}$ 
\begin{equation}
    \int_{\partial N}Q_{\chi}(\omega ',\omega)=\frac{1}{32\pi^2}\int_{\partial N}\epsilon_{abcd}\left(2\xi_b^a \wedge \Omega_d^c -\frac{4}{3} \xi_b^a \wedge \xi_e^c \wedge \xi_d^e \right)
\end{equation}
 and if $\Bar{M}=M\cup \partial\Bar{M}$, where $(M,g)$ is a noncompact complete Riemannian manifold with fibered geometry at infinity and $N=M_{\epsilon}=\{x\geq \epsilon\}$ then we have
 \begin{equation}\label{EQ-36}
     \lim_{\epsilon \to 0}\int_{\partial M_{\epsilon}}Q_{\chi}=\frac{1}{2\pi^2}\text{vol}(\partial \Bar{M})+\alpha(\partial \Bar{M})
 \end{equation}
where 
\begin{equation}
    \alpha(\partial \Bar{M})=\frac{1}{8\pi^2}\int_{\partial \Bar{M}}\epsilon_{abc}\omega^a \wedge\left(\Omega_c^b-\omega^b\wedge \omega^c\right)
\end{equation}
For the above discussion, see \cite{CHERN1, EGH1}.

Let $(M,g)$ be a 4-manifold that possesses $L^2$ solution $(A,\phi)$ to the Seiberg-Witten equations. Later in this paper, we will discuss the Seiberg-Witten equations. Di Cerbo \cite{DCERBO} proved an $L^2$ analogue of an estimate proved by LeBrun \cite{LEB1, LEB2, LEB3}. The version of the theorem proved by Di Cerbo is as follows
\begin{theorem}[Di Cerbo]
    Let $(M,g)$ be a closed 4-manifold and $(A,\phi)$ be an $L^2$ solution of the Seiberg-Witten equation. Then 
    \begin{equation}\label{EQ-38}
    \frac{1}{32\pi^2}\int_Xs_g^2dvol_g\geq c_1^2(\mathfrak{s})    
    \end{equation}
    where $s_g$ is the scalar curvature, $\mathfrak{s}$ is the $Spin^c$ structure on $M$, and $c_1(\mathfrak{s})$ is the $L^2$ chern class.
\end{theorem}
\begin{proof}
    For detailed proof, see \cite{DCERBO}.
\end{proof}
If the fiber is a single point, then $B=\partial \Bar{M}$ and the fibered boundary metric
$$g\sim dt^2+t^2g_{\partial \Bar{M}},$$ where $t$ can be thought of as the distance from the base point. In this case, the geometry at infinity is asymptotically conical and we say that the manifold $M$ is asymptotic to a cone over $(\partial \Bar{M},g_{\partial \Bar{M}})$. Now, we are ready to prove the main result as follows
\begin{theorem}
    Let $(M,g)$ be a complete 4-manifold which is asymptotic to a cone over $(\partial \Bar{M},g_{\partial \Bar{M}})$ and suppose that $(A,\phi)$ is an $L^2$ solution to the Seiberg-Witten equations. Then 
    \begin{eqnarray*}
        2\chi(M) - 3\sigma(M) \geq \int_{M}\left[|W|^2 - \left( |W^+|^2 - |W^-|^2\right) - |Z|^2\right]\text{d}vol+\frac{3}{2}\eta(\partial\Bar{M})+\frac{1}{3}c_1^2(\mathfrak{s}) \\
       - \left[\frac{1}{\pi^2}vol(\partial\Bar{M}+2\alpha(\partial \Bar{M})\right]
    \end{eqnarray*}
    
    where $\mathfrak{s}$ is the $Spin^c$ structure on $M$, and $c_1(\mathfrak{s})$ is the $L^2$ chern class.
\begin{proof}
   Suppose $\Bar{M}=M\cup \partial \Bar{M}$ is compactification $M$ and $M_{\epsilon}=\{x\geq \epsilon\}$. Then from (\ref{EQ-3.3}) and (\ref{EQ-3.4}), we have
    \begin{eqnarray}
        \chi(M_{\epsilon})=\int_{M_{\epsilon}}\text{Pf}\left(\frac{\Omega}{2\pi}\right) - \int_{\partial M_{\epsilon}} Q_{\chi}(\omega ',\omega), \\
        \sigma(M_{\epsilon})=\int_{M_{\epsilon}} L\left(\frac{\Omega}{2\pi}\right) - \frac{1}{2}\eta(\partial M_{\epsilon}) - \int_{\partial M_{\epsilon}} Q_{\sigma}(\omega ',\omega).
    \end{eqnarray}
    Then using Lemma \ref{LM-3.1} and taking the limit $\epsilon \to 0$, equations (3.9) and (3.10) reduce to
    \begin{eqnarray}
      \label{EQ-311}  \sigma(M)=\int_M L\left(\frac{\Omega}{2\pi}\right) - \frac{1}{2}\eta(\partial \Bar{M})\\
      \label{EQ-312}  \chi(M)=\int_M\text{Pf}\left(\frac{\Omega}{2\pi}\right) - \lim_{\epsilon \to 0} \int_{\partial M_{\epsilon}} Q_{\chi}(\omega ',\omega)
    \end{eqnarray}
    Again using (\ref{EQ-26}) and (\ref{EQ-27}), equations (\ref{EQ-311}) and (\ref{EQ-312}) become
   \begin{eqnarray}
      \label{EQ-313}  3\left[\sigma(M)+ \frac{1}{2}\eta(\partial \Bar{M})\right]=\frac{1}{4\pi^2}\int_M \left(|W^+|^2-|W^-|^2\right) d\text{vol}, \\
      \label{EQ-314}  2\left[\chi(M)+ \lim_{\epsilon \to 0} \int_{\partial M_{\epsilon}} Q_{\chi}(\omega ',\omega)\right]=  \frac{1}{4\pi^2}\int_M \left(|W|^2 - |Z|^2+\frac{1}{24}S^2\right)d\text{vol}.
    \end{eqnarray} 
    Subtracting (\ref{EQ-313}) from (\ref{EQ-314}) and using (\ref{EQ-36}), (\ref{EQ-38}), we have
    \begin{eqnarray*}
        2\chi(M) - 3\sigma(M) \geq \int_{M}\left[|W|^2 - \left( |W^+|^2 - |W^-|^2\right) - |Z|^2\right]\text{d}vol+\frac{3}{2}\eta(\partial\Bar{M})+\frac{1}{3}c_1^2(\mathfrak{s}) \\
       - \left[\frac{1}{\pi^2}vol(\partial\Bar{M}+2\alpha(\partial \Bar{M})\right].
    \end{eqnarray*}
\end{proof}
\begin{corollary}
     Let $(M,g)$ be a complete Einstein 4-manifold which is asymptotic to a cone over $(\partial \Bar{M},g_{\partial \Bar{M}})$ and suppose that $(A,\phi)$ is a $L^2$ solution to the Seiberg-Witten equations. Then 
    \begin{eqnarray*}
        2\chi(M) - 3\sigma(M) \geq \frac{3}{2}\eta(\partial\Bar{M})+\frac{1}{3}c_1^2(\mathfrak{s}) - \left[\frac{1}{\pi^2}vol(\partial\Bar{M}+2\alpha(\partial \Bar{M})\right].
    \end{eqnarray*}
\end{corollary}
\begin{proof}
    For Einstein manifolds, $|W|^2=|W^+|^2+|W^-|^2$ and $Z=0$, hence the result follows from above theorem.
\end{proof}
\end{theorem}
\section{Concluding Remarks}
In the introduction, we discuss many generalizations of the Hitchin-Thorpe inequality and their applications to determine many geometric properties of a space. Then we discussed on Seiberg-Witten gauge theory, $L^2$-cohomology, and $ L^2$-Chern-Weil theory. Next, we discussed on fibered geometry at infinity and derived expression of Euler number and signature of these spaces using the Chern-Simons correction term. Then we presented some results of Seiberg-Witten gauge theory. In the end, we derived a result using gauge theoretic tools and fibered geometry at infinity.\\
\vskip6pt
%\noindent \textbf{Acknowledgement:} 
%\\ \vskip4pt
\noindent \textbf{Data availibility:} The authors confirm that the data supporting the findings of this study are included within the article.\\ \vskip2pt
\noindent \textbf{Conflict of interest:} The authors confirm that there is no conflict of interest.

\vspace{.25in}

\noindent{Shyamal Kumar Hui \\ Professor, Department of Mathematics, The University of Burdwan, Golapbag, Burdwan 713104, West Bengal, India}\\
Email: \texttt{skhui@math.buruniv.ac.in}\\

\noindent{Chandan Kumar Adak \\ Research Scholar, Department of Mathematics, The University of Burdwan, Golapbag, Burdwan 713104, West Bengal, India\\
\&\\
Assistant Professor, Department of Mathematics, Asansol Girls' College, Asansol 713304, West Bengal, India}\\
Email: \texttt{chandankradak@gmail.com}

\end{document}